\documentclass[11pt]{amsart}
\usepackage[english]{babel}
\usepackage[a4paper,margin=1in]{geometry}
\usepackage[T1]{fontenc}
\usepackage[utf8]{inputenc}
\usepackage{amsmath,amssymb,amsthm}
\usepackage{mathtools}
\usepackage{enumitem}
\usepackage[colorlinks=true,linkcolor=blue,citecolor=blue,urlcolor=blue]{hyperref}

\hypersetup{
  colorlinks=true,
  linkcolor=blue,
  citecolor=blue,
  urlcolor=blue
}

\newtheorem{theorem}{Theorem}
\newtheorem{lemma}[theorem]{Lemma}
\newtheorem{definition}[theorem]{Definition}

\newtheorem{question}[theorem]{Question}

\begin{document}


\title[AEWG r.e. sets]{Almost-Everywhere Computation of Weak Generics Relative to R.E. Sets}

\author{Xuanheng Zhao}
\address{School of Mathematics\\
 Nanjing University\\
 Nanjing, Jiangsu 210093, People's Republic of China}

\email{xuanheng21@gmail.com}

\begin{abstract}
  A set $B$ is AEWG if almost every set computes a weakly
  1-$B$-generic. Extending two results of Hirschfeldt, Jockusch and Schupp, we prove that there is an AEWG superhigh r.e. set and every $\text{low}_2$ r.e. set is AEWG.
\end{abstract}

{\maketitle}

\section{Introduction}

The main results in this paper are pure r.e. degree theoretical statements.
However, the motivation comes from coarse computability, a new area that has
received attention of mathematical logicians in the last two decades.
We say that two sets $A, B \subseteq \omega$ are {\emph{coarsely similar}}, if their symmetric difference has asymptotic density 0. Coarse similarity is an equivalence relation on $\mathcal{P} (\omega)$.
Let $\mathcal{S}$ denote the set of all coarse similarity classes and write
$[A]$ for the coarse similarity class of the set $A$. 

Hirschfeldt, Jockusch
and Schupp {\cite{MR4741732}} investigated the natural metric $\delta ([A], [B])$ on $\mathcal{S}$, which is equal to the upper asymptotic density of $\{ m: A (m)
   \neq B (m) \}$. They then defined a metric $H$
between Turing degrees as follows:
Given a Turing degree $\mathbf{d}$, a set $A \subseteq \omega$ is
{\emph{coarsely $\mathbf{d}$-computable at density}} $r \in [0, 1]$ if there
is a set $C \leq_T \mathbf{d}$ such that the lower asymptotic density of the set $\{m:A(m)=C(m)\}$ is no less than $r$.
The {\emph{coarse $\mathbf{d}$-computability bound}} of a set $A$ is
 the supremum of all reals $r$ for which $A$ is coarsely $\mathbf{d}$-computable at density $r$.
Let $\bar{\mathbf{d}} = \{ [A] : \gamma_{\mathbf{d}} (A) = 1 \}$ for any
degree $\mathbf{d}$. The {\emph{Hausdorff distance}}, $H (\mathbf{d}, \mathbf{e})$,
between the degrees $\mathbf{d}$ and $\mathbf{e}$ is the Hausdorff distance of $\mathbf{d}$ and $\mathbf{e}$ in $\mathcal{S}$.

It was shown in {\cite{MR4741732}} that $H$ is a $\left\{ 0, \frac{1}{2}, 1 \right\}$-valued metric between degrees, which is a relativization of Monin's Theorem {\cite{MR3883780}}.
A degree $\mathbf{a}$ is {\emph{dispersive}}
if $H (\mathbf{a}, \deg (B)) = 1$ for almost every set $B$, where $\deg (B)$
is the Turing degree of $B$. A degree $\mathbf{a}$ is {\emph{attractive}} if
$H (\mathbf{a}, \deg (B)) = \frac{1}{2}$ for almost every set $B$. The
Kolmogorov 0-1 Law (see {\cite{MR2732288}}) implies that any measurable Turing
invariant set (a collection of sets that is closed under Turing equivalence)
has measure 0 or 1. In particular, every Turing degree is either dispersive or
attractive. Certain classes of degrees are proved to be contained in the
dispersive or attractive degrees.

\begin{theorem}[{\cite{MR4741732}}]
  \label{1}
  \begin{enumerate}
  
      \item There is a dispersive high r.e. degree.
      \item Every low r.e. degree is dispersive.
      \item The class of dispersive degrees is closed downwards.
      \item Every almost everywhere dominating degree is attractive.
  \end{enumerate}
\end{theorem}

A set $W \subseteq 2^{< \omega}$ is {\emph{dense}} if for every string
$\sigma$, there is a string $\tau \in W$ extending $\sigma$. Given $n \geq 1$,
a set $A$ is {\emph{weakly n-generic}} if for every dense $\Sigma^0_n$ set $W
\subseteq 2^{< \omega}$ there is $m$ such that $A \upharpoonright m \in W$
(sometimes we write $A \in [W]$). The definition can be relativized to any
oracle. For example, a set $A$ is {\emph{weakly 1-$Z$-generic}} if for
every dense $\Sigma^{0}_1(Z)$ set $W \subseteq 2^{< \omega}$ there is $m$ such
that $A \upharpoonright m \in W$.

\begin{definition}
  A set $B$ is {\emph{AEWG}} if almost every set computes a set
  that is weakly 1-$B$-generic. A degree $\mathbf{b}$ is AEWG if there is an AEWG
  set $B \in\mathbf{b}$.
\end{definition}

It was pointed out in \cite{MR4741732} that for a degree, being AEWG implies being dispersive. But it is still unknown whether being dispersive implies being AEWG even within the r.e. degrees.
 \begin{question}[{\cite[part of Question 9.2]{MR4741732}}]\label{q1}
     Can a set $A$ be dispersive without it being AEWG?
 \end{question}
 
 In fact, Theorem \ref{1}(1) and (2) were derived as a corollary from the following theorem.

\begin{theorem}[\cite{MR4741732}]\label{i}
  There is an AEWG high r.e. set and every $\text{low}$ r.e. set $A$ is AEWG.
\end{theorem}
 
 Dobrinen and
Simpson {\cite{MR2078930}} defined a degree $\mathbf{a}$ to be {\emph{almost
everywhere dominating (a.e.d.)}} if for almost every set $B$, every
$B$-recursive function $f$ is dominated by an $\mathbf{a}$-recursive function
$g$, that is $(\exists m) (\forall n > m) [f (n) < g (n)]$. A set $A$ is {\emph{superhigh}} if $A'
\geq_{\text{tt}} \emptyset''$. Simpson
{\cite{MR2351944}} showed that every a.e.d degree is superhigh and there
exists an r.e. set which is superhigh and not a.e.d. See Nies
{\cite{MR2548883}} page 361-362 for a detailed diagram of the implications of
some upward (downward) closed properties for Turing degrees and r.e. degrees
including the properties mentioned above.

The following question was asked in \cite{MR4741732}: Do the notions of being attractive and being
a.e.d. coincide for r.e. degrees? Recently, Hirschfeldt and Royer \cite{HR} showed that if a degree is not diagonally non-recursive, then it is attractive iff it is a.e.d. and hence answered the question in the affirmative. However, their proof does not use the AEWG property, but rather works directly with Monin's notion of $2^{2^n}$-infinitely often equality. It is still unknown whether being dispersive implies being AEWG for an r.e. degree. So the following question is the next step towards Question \ref{q1}: \emph{If $A$ is an r.e. set which is not a.e.d., must $A$ be AEWG?}
In view of this, we obtain the following results.

\begin{theorem}
  \label{main}There is an AEWG superhigh r.e. set.
\end{theorem}

\begin{theorem}
  \label{low2}Every $\text{low}_2$ r.e. set is AEWG.
\end{theorem}

The two
proofs of are very different in the way of arranging the
tree of strategies.

Ng {\cite{MR2387945}} constructed an r.e. minimal pair of superhigh degrees (independently by Shore (unpublished)). We follow the framework in \cite{MR2387945} for constructing a superhigh r.e.
set in Theorem \ref{main} in Section 2.

In the proof of Theorem \ref{low2} in Section 3 we use the ``certification process''
developed recently by Cholak, Downey and Greenberg {\cite{CDG}}. 

We assume the familiarity with recursion theory as in {\cite{MR882921}} and early chapters in
{\cite{MR2732288}}, for instance. We adopt the convention of using uppercase
Greek letters for functionals, and lowercase Greek letters for their use. We
append $[s]$ to functionals or their use to
describe their values at stage $s$.

\section{Proof of Theorem \ref{main}}
\subsection{Requirements}\label{requ}
To show that the r.e. set $A$ to be constructed is superhigh, it is enough to
build a truth-table reduction $A' \geq_{\mathrm{tt}} \mathrm{Tot}
   :=\{e:\Phi_e \text{ is total}\}$.
We divide this reduction into positive requirements $P_e$: $ e\in\mathrm{Tot}\Longleftrightarrow A'\vDash \sigma_e$,
where the truth tables $\sigma_e$ are defined uniformly in the
construction.  We also build a functional $\Psi$ and meet the negative
requirements
\[
Q_e:\quad W_e^A\text{ dense }\Rightarrow
(\forall X)\bigl(\Psi^X\text{ total }\rightarrow
       (\exists n)\,\Psi^X\upharpoonright n\in W_e^A\bigr).
\]
Finally we meet the global requirement $M$: the set of oracles $X$ such that $\Psi^X$ is total has positive measure. If all $Q_e$ and $M$ are satisfied, then on a positive-measure class the value
$\Psi^X$ is weakly $1$-$A$-generic.  The class of oracles computing a weakly
$1$-$A$-generic is Turing invariant, so the Kolmogorov 0-1 law raises this
positive measure to measure one.

\subsection{The basic module}\label{2.2}
The usual basic module for $Q_e$ as in \cite{MR4741732} fixes a parameter $k=k_e$ and tries, one by one, cones
$[\rho]$ with $|\rho|=k$.  While a cone is claimed, the construction keeps
$\Psi$ constant along that cone until all $\Psi$-values of current leaves of the cone have
extensions in $W_e^A$; at that point it extends $\Psi$
into $W_e^A$ on the whole claimed piece and moves to the next cone.

The point of the present version is that a higher-priority node may already
claim some proper subcones of $[\rho]$.  If these subcones do not cover
$[\rho]$, the node for $Q_e$ claims the clopen remainder.  Thus a claim is not
just a string, but a pair $(\rho,F)$,
where $|\rho|=k$, $F$ is a finite prefix-free set of proper extensions of
$\rho$, and $R(\rho,F):=[\rho]\setminus\bigcup_{\xi\in F}[\xi] \neq \emptyset$.
We call $\rho$ the \emph{base} of the claim and $R(\rho,F)$ the
\emph{claimed clopen set}.  At a later stage $s$, the module waits until every
string $\tau$ of length $s-1$ with $[\tau]\subseteq R(\rho,F)$ has an extension
of $\Psi^\tau$ in $W_e^A$.  It then defines $\Psi^{\tau 0}=\Psi^{\tau 1}$ to be
such an extension, for all these $\tau$, and drops the claim.  The excluded
subcones are left to the higher-priority nodes which already claimed them. If
there is no such a stage $s$, we may believe that $W_e^A$ is not dense. And if we carefully define the parameters $k_e$, the waste of measure in the case that there is no such a stage $s$ is tame, so we meet requirement $M$.

\subsection{The tree}
As usual, our construction is performed on a tree $T$ of
strategies. Nodes on $T$ represent guesses about
the eventual behaviour of some aspects of the construction, and use their
guesses to meet a requirement they are assigned to. We will define $T$ to be a finitely branching recursive tree (but not the binary tree). The
outcomes at odd levels are $\infty$ and $f$, ordered $\infty<_{\mathrm{left}}f$.
The outcomes at even levels are finite initial segments of $\omega$, ordered in
the usual way.  We use the notation $\alpha<_{\mathrm{left}}\beta$ and
$\alpha<_L\beta$ in the standard sense: $\alpha<_L\beta$ means that $\alpha$ is
either strictly to the left of $\beta$ or is a proper initial segment of
$\beta$.
Nodes of length $2e+1$ are \emph{odd nodes} and are assigned to $P_e$.  Each odd node
$\gamma$ has two \emph{boxes} (see subsection \ref{not}), denoted by $\eta_\infty^\gamma$ and $\eta_f^\gamma$.
Nodes of length $2e$ are \emph{even nodes} and are assigned to $Q_e$. Each even node
$\alpha$ has a \emph{parameter} $k_\alpha>1$ and at any stage has at most one current
claim $(\rho,F)$ and $|\rho|=k_\alpha$. A node $\alpha$ is said to be an {\emph{$R$-node}}, if it is assigned to the
requirement $R$.

The parameters and even outcomes are defined recursively as follows.  Put
$k_\lambda=2$.  Suppose the parameter has been defined for all even nodes of
length at most $2n$, and let $\sigma_0,\ldots,\sigma_{m-1}$ be the nodes of
length $2n$.  The outcomes of $\sigma_i$ are $0,1,2,\ldots, 2^{k_{\sigma_i}+1}+1$,
where even outcomes correspond to expansionary action and odd outcomes to
non-expansionary waiting.  This defines level $2n+1$.  The nodes of length
$2n+2$ have the form $\sigma x$, where $|\sigma|=2n+1$ and
$x\in\{\infty,f\}$.  List them from left to right as
$\tau_0,\ldots,\tau_j$.  Choose the least $c>2$ which is not the parameter of
any shorter even node and set $k_{\tau_i}=c+i$.

At each stage $s$ we define an approximation to the true path, $\delta_s$ of length $s$. A node $\alpha$ is {\emph{visited}} at stage $s$, if
$\alpha \subseteq \delta_s$. The leftmost path on $T$ such that every initial segment
is visited infinitely often is the {\emph{true path}}. Since $T$ is finitely branching, the true path exists. At each stage $s$
we will define $\Psi^{\tau}$ for all $\tau \in 2^s$.

\subsection{Boxes and the truth table}\label{not} The concepts in this subsection originate from Ng {\cite{MR2387945}}. We define an auxiliary
functional $\Theta$ in the construction. Fix a partial recursive 1-1 function $(\gamma,x)\mapsto \eta_x^{\gamma}$ defined for $\gamma$ odd and $x \in \{\infty,f\}$ such that $\Theta^A (\gamma,x) \downarrow$ if and only if $A' (\eta_x^{\gamma}) = 1$.
We will not use the function $(\gamma,x)\mapsto \eta_x^{\gamma}$ in the construction. $\eta_x^{\gamma}$ is just the name of a box by the follows.

\begin{definition}
  Given $s$, an odd $\gamma \in T$ and $x \in \{ \infty, f \}$, we say that at stage $s$:
\begin{enumerate}[label=(\roman*)]
    \item The box $\eta_x^{\gamma}$ is {\emph{empty}} if $\Theta^{A} (\gamma,x)[s] \uparrow$.
    \item To {\emph{fill}} the empty box $\eta^{\gamma}_x$ with use $u$ means
  that we define $\Theta^{A \upharpoonright u + 1} (\gamma,x)[s] \downarrow$ with use
  $u$.
  \item To {\emph{clear}} the box $\eta^{\gamma}_x$ which has already been
  filled with use $u$ at some stage $t<s$ such that $A_t \upharpoonright u + 1 
  = A_s \upharpoonright u + 1$ means that we enumerate $u$ into $A$. To clear an empty box means doing nothing.
\end{enumerate}
\end{definition}

Now fix $e$ and list the nodes of length $2e+1$ from left to right as $\gamma_0<_{\mathrm{left}}\cdots
   <_{\mathrm{left}}\gamma_n$.
Put $\rho_i=\gamma_{n-i}$.  The truth table $\sigma_e$ reads this level in the order $\eta_f^{\rho_0},\eta_\infty^{\rho_0},
  \eta_f^{\rho_1},\eta_\infty^{\rho_1},\ldots,
  \eta_f^{\rho_n},\eta_\infty^{\rho_n}$.
It ignores everything after the first occupied box.  Its value is $1$ exactly
when this first occupied box is an $\eta_\infty$-box.  Equivalently,
\[
\sigma_e=
\bigvee_{i\le n}
  \left(
    A'(\eta_\infty^{\rho_i})\wedge
    \neg A'(\eta_f^{\rho_i})\wedge
    \bigwedge_{j<i}
      \bigl(\neg A'(\eta_f^{\rho_j})\wedge
             \neg A'(\eta_\infty^{\rho_j})\bigr)
  \right).
\]
Since the tree is recursive and each level is finite, the sequence
$\{\sigma_e\}_{e\in\omega}$ is recursive.

\subsection{Legal residual claims and expansionary stages}
At a stage, if an even node $\beta$ has a current claim, let $b(\beta)$ denote
its base.  For an even node $\alpha$ and a string $\rho$ of length
$k_\alpha$, let
\[
   H_s(\alpha)=\{b(\beta):\beta<_L\alpha\text{ and }\beta
          \text{ has a current claim at stage }s\}.
\]
We say that $\rho$ is \emph{legal for $\alpha$ at stage $s$} if no member of
$H_s(\alpha)$ is an initial segment of $\rho$, and if the prefix-minimal
members of $\{\xi\in H_s(\alpha):\rho\subset \xi\}$
do not cover $[\rho]$.  If $\rho$ is legal, let $F_s(\alpha,\rho)$ be this
finite prefix-minimal set.  Then $R_s(\alpha,
\rho)=[\rho]\setminus
       \bigcup_{\xi\in F_s(\alpha,
\rho)}[\xi]$
is the clopen remainder which $\alpha$ may claim.  When $\alpha$ claims
$\rho$, it records the pair $(\rho,F_s(\alpha,
\rho))$; the recorded set $F$ is
not changed until the claim is dropped or $\alpha$ is initialized.
If $C=R(\rho,F)$ is a current claim, say
that $\tau$ is \emph{inside} $C$ if $[\tau]\subseteq C$. 

\begin{definition}
    Each even node $\alpha$ carries a \emph{restraint} $r_\alpha$,
initialized to $-1$.  After the last initialization of $\alpha$,
$r_\alpha$ is the maximum of the oracle uses of all $W_e^A$-enumerations which
have been used by $\alpha$ in successful extensions of $\Psi$.
\end{definition}

\begin{definition}\label{d6}
    Let $\alpha$ be an even node of length $2e$.  A stage $s>0$
is \emph{$\alpha$-expansionary} if:

\begin{enumerate}[label=(\roman*)]
\item $\alpha$ is visited at stage $s$;
\item $\alpha$ has a current claim $C=R(\rho,F)$ and
      $s-1\geq \max\{ |\xi|:\xi\in F\cup\{\rho\}\}$;
\item for every $\tau\in 2^{s-1}$ inside $C$, there is a string in
      $W_e^A[s]$ properly extending $\Psi^\tau$.  Among all stage-$s$
      computations witnessing such an extension, effectively choose one of least oracle
      use, and denote the extension
      by $\nu_\tau$;
\item if $w$ is the maximum of the oracle uses of the enumerations chosen in
      (iii), and $ w^+=\max\{w,r_\alpha\}$,
      then 
      \begin{enumerate}[label=(\alph*)]
      \item for every odd $\beta$ with $\beta\infty\subseteq\alpha$, the box
            $\eta_f^\beta$ is empty or has use greater than $w^+$;
      \item for every odd $\gamma>_{\mathrm{left}}\alpha$, both
            $\eta_f^\gamma$ and $\eta_\infty^\gamma$ are empty or have use
            greater than $w^+$.
      \end{enumerate}
\end{enumerate}
\end{definition}

\begin{definition}
    For an odd node $\beta$ of length $2e+1$, define $l(\beta,s)<s$ to be the maximal $y$ such that $\Phi_e(x)[s]\downarrow$ for all $x<y$.
A stage $s>0$ is \emph{$\beta$-expansionary} if $\beta$ is visited at $s$ and
$l(\beta,s)>l(\beta,s^-)$, where $s^-<s$ is the last stage at which
$\beta$ was visited. If $s$ is the first stage at which $\beta$ is visited, we assume $l(\beta,s^-)=0$ by convention.
\end{definition}

If $l (\beta, s)$ constantly increases, then we have reason
to believe that $\Phi_e$ is total. 

The key device in the construction is an \emph{outcome jump} at even nodes.
After the last initialization of an even node $\alpha$, its visits are
organized in adjacent pairs $c,c+1$, where $c$ is even: $c$ is the
expansionary outcome, while $c+1$ is the corresponding waiting outcome.
Suppose that, at an $\alpha$-expansionary stage, the computations which
$\alpha$ is about to preserve have use bound $w^+$, but a box already
occupied at an odd proper extension of $\alpha$ has use at most $w^+$.
Instead of returning to outcome $c$, the construction advances to the next
even outcome $c+2$; relative to the immediately preceding outcome this is a
jump of two from $c$, or of one from $c+1$. Every such box lies on an earlier
branch whose $\alpha$-outcome is below $c+2$, and the jump therefore places it
strictly to the left of the new path approximation. Unless a node to the left
of $\alpha$ is later visited or $\alpha$ is later initialized, all subsequent
outcomes at $\alpha$ are at least $c+2$, so this box can no longer be cleared
by either end-of-stage clearing operation. Any box filled after the jump
receives a fresh use greater than $w^+$ and is harmless to the protected
computations. Thus the jump disposes of every existing small-use threat; the
only possible exceptions are precisely the higher-priority injuries excluded
once the true-path node has stabilized.

\subsection{The construction}
At stage $0$, define $\Psi^\lambda=\lambda$ and regard every even node as
\emph{initialized} for convenience (to
initialize an even node we mean we delete its claim if it has a current claim,
declare it to be not satisfied if it is currently satisfied, and set its restraint $r_\alpha$ to $-1$). 

At stage $s>0$, we inductively define $\delta_s$ of length $s$
and let the initial segments of $\delta_s$ to act. Suppose we have defined
$\alpha = \delta_s \upharpoonright 2 e$ for $2 e < s$. We allow the even nodes to act as soon as we see that it is on
$\delta_s$.

\textbf{Case 1.}  $\alpha$ was not initialized earlier at this stage, is not
satisfied, and $s$ is $\alpha$-expansionary.  Let 
$C=R(\rho,F)$ and $w^+$ be as in Definition \ref{d6}.

\textbf{Subcase 1-1.}  There is no odd $\gamma\supset\alpha$ and
$x\in\{\infty,f\}$ such that the box $\eta_x^\gamma$ has use at most $w^+$.
Proceed as follows.

\begin{enumerate}[label=(\arabic*)]
\item For every $\tau\in2^{s-1}$ inside $C$, choose $\nu_\tau$ as in Definition \ref{d6} and
      define $\Psi^{\tau0}=\Psi^{\tau1}=\nu_\tau$.  After these extensions have
      been made, set the restraint $r_\alpha=w^+ $.
  
\item Drop the claim $C$.  Then let $\alpha$ claim the
      $<_{\mathrm{left}}$-least string $\rho'$ of length $k_\alpha$ which is
      legal for $\alpha$ currently and has not been used as a base
      by $\alpha$ since its last initialization. If there is no
      such $\rho'$, declare $\alpha$ is \emph{satisfied}.
\item Initialize all even $\beta\supset\alpha$.
\item Let $c<2^{k_\alpha+1}+1$ be the even number such that, at the last visit to $\alpha$ since its last initialization (such a stage exists since $\alpha$ has a claim currently), the outcome of
      $\alpha$ was $c$ or $c+1$.  Define $\delta_s(2e)=c$.
\end{enumerate}

\textbf{Subcase 1-2.}  Some odd $\gamma\supset\alpha$ has a box
$\eta_x^\gamma$ of use at most $w^+$.  Perform steps (1)-(3) from Subcase 1-1.
Let $c$ be as in Subcase 1-1(4) and define $\delta_s(2e)=c+2$.  This
is always possible since $\alpha$ can have at most $2^{k_\alpha}$ expansionary
actions between two initializations.

\textbf{Case 2.}  $\alpha$ was not initialized earlier at this stage, is not
satisfied, has a current claim, and $s$ is not $\alpha$-expansionary.  Let $c$ be as in Subcase 1-1(4) and define $\delta_s(2e)=c+1$.

\textbf{Case 3.}  $\alpha$ was not initialized earlier at this stage, is not
satisfied, and has no current claim (this could happen if $s$ is
the first stage at which $\alpha$ is visited since the last stage at which
$\alpha$ was initialized).  Let $\alpha$ claim the
$<_{\mathrm{left}}$-least string $\rho$ of length $k_\alpha$ which is legal for
$\alpha$ at this stage and has not been used as a base by $\alpha$ since the
last initialization of $\alpha$.  If no such $\rho$ exists, declare $\alpha$
satisfied.  Initialize all even $\beta\supset\alpha$, 
and define $\delta_s(2e)=1$ if a claim was made
and $\delta_s(2e)=0$ otherwise.

\textbf{Case 4.}  $\alpha$ was not initialized earlier at this stage and is
satisfied.  If $\alpha$ became satisfied in Case 3 without making a claim, define the outcome to be 0.
Otherwise let $c$ be as in Subcase 1-1(4) and define $\delta_s(2e)=c$.

\textbf{Case 5.}  $\alpha$ was initialized earlier at this stage.  Define
$\delta_s(2e)=1$.

We say that an even node \emph{acts} at stage $s$ if Case 1 or Case 3 holds.

Now suppose $\beta=\delta_s\upharpoonright(2e+1)$ has been defined and
$2e+1<s$.  Define $\delta_s(2e+1)=\infty$ if $s$ is
$\beta$-expansionary, and define $\delta_s(2e+1)=f$ otherwise.
After $\delta_s$ of length $s$ has been defined, perform the following end-of-stage actions.

\begin{enumerate}[label=(\roman*)]
\item Initialize all even $\beta>_{\mathrm{left}}\delta_s$.
\item For every odd $\gamma>_{\mathrm{left}}\delta_s$, clear
      $\eta_\infty^\gamma$ and $\eta_f^\gamma$.
\item For every odd $\alpha$ such that $\alpha\infty\subseteq\delta_s$, clear $\eta_f^\alpha$.
\item For every odd $\alpha\subset\delta_s$, in increasing order of length,
      fill $\eta_{\delta_s(|\alpha|)}^\alpha$ if it is empty, with the use a fresh
      number larger than every number which has appeared so far.
\item For each $\mu\in2^s$ for which $\Psi^\mu$ has not yet been defined, if
      $[\mu]$ is contained in one of the current claimed clopen sets, define $\Psi^\mu=\Psi^{\mu\upharpoonright(s-1)}$. Otherwise define $\Psi^\mu=\Psi^{\mu\upharpoonright(s-1)}0$.
    
\end{enumerate}

\subsection{Verification}

\begin{lemma}\label{claims-antichain-main}
At the end of every stage, the bases of distinct current claims have pairwise
distinct lengths, and the clopen sets currently claimed by even nodes are pairwise
disjoint.
\end{lemma}

\begin{proof}
The recursive assignment gives distinct parameters to distinct
even nodes. Since a node only uses bases of length
equal to its parameter, the base lengths of distinct current claims are
pairwise distinct.
For disjointness, consider the stage at which an even node $\alpha$ makes a new
claim with base $\rho$ and exclusion set $F$.  Any node extending $\alpha$ is
initialized immediately, and any node to the right of the final approximation
is initialized at the end of the stage.  Thus it is enough to compare the new
claim with claims of nodes $\beta<_L\alpha$.

Let $\eta=b(\beta)$.  If $\eta$ is an initial segment of $\rho$, then $\rho$ is
not legal, so this case cannot occur.  If $\rho$ is an initial segment of
$\eta$, then either $\eta$ is itself a member of $F$, or it extends a
prefix-minimal member of $F$; in either case the new claimed clopen set excludes
$[\eta]$ and hence is disjoint from the claim of $\beta$.  If $\rho$ and
$\eta$ are incomparable, the corresponding cylinders are already disjoint.
\end{proof}

\begin{lemma}\label{box-nesting-main}
Suppose that at the end of a stage a box attached to an odd node $\gamma$ is
nonempty.  If $\rho$ is an odd initial segment of $\gamma$, and
$\rho x\subseteq\gamma$ for $x\in\{\infty,f\}$, then the box $\eta_x^\rho$ is
nonempty and its use is smaller than the use of the box at $\gamma$.
\end{lemma}

\begin{proof}
Boxes are filled only after all clearings for the stage have been performed,
and they are filled along the current approximation from top to bottom with
fresh uses.  Hence, at the stage when a box at $\gamma$ is filled, each relevant
predecessor box is already nonempty and has smaller use.  If later an
enumeration into $A$ destroys the predecessor computation, then it also destroys
the computation at $\gamma$, whose use is larger.  Thus the relation holds at
every stage.
\end{proof}

\begin{lemma}\label{finite-injury-main}
Let $\rho$ be an even node on the true path. Then $\rho$ is initialized only finitely many times, and after its last initialization it acts only finitely many
times.
\end{lemma}

\begin{proof}
The proof is by induction on the length of $\rho$.  Once no node to
the left of $\rho$ is visited, only proper
initial segments of $\rho$ can initialize it; by induction, those nodes act
only finitely often.

Assume now that $\rho$ is on the true path and work after its last
initialization.  No higher-priority node changes its claim after this point.  Hence the set of legal bases of length
$k_\rho$ is fixed.  The node $\rho$ never uses the same base twice between two
initializations, and there are only $2^{k_\rho}$ possible bases.  Case 3 can
occur only when $\rho$ has no current claim, and each later Case 1 drops one
claim.  Hence $\rho$ acts only finitely many times.
\end{proof}

\begin{lemma}\label{extension-jump-main}
Let $\alpha$ be an even node and suppose that, after the last initialization of
$\alpha$, Subcase 1-2 is applied to $\alpha$ at a stage $s$.  Let
$w^+$ be the corresponding bound, and suppose that the outcome assigned to
$\alpha$ at stage $s$ is $c+2$.  Assume further that after stage $s$ no node to
the left of $\alpha$ is visited and $\alpha$ is not initialized.  Then no box
attached to an odd proper extension of $\alpha$ which is nonempty at the
beginning of the end-of-stage actions at stage $s$ and has use at most $w^+$ is
ever cleared after stage $s$.
\end{lemma}

\begin{proof}
Let $\eta_x^\gamma$ be a box as in the statement.  Since it is already nonempty
before the end-of-stage actions at stage $s$, it was filled at an earlier stage
at which the approximation to the path passed through $\gamma$. Then the outcome of $\alpha$ on that earlier path
was strictly smaller than $c+2$.  Hence, once stage $s$ assigns the outcome
$c+2$ to $\alpha$, the node $\gamma$ is strictly to the left of the current
approximation below $\alpha$.
At every later stage at which $\alpha$ is visited, the outcome of $\alpha$ is
again at least $c+2$, unless $\alpha$ has first been initialized; the latter is
excluded by hypothesis.  Thus $\gamma$ remains strictly to the left of the later
approximations whenever the construction passes through $\alpha$.  The ordinary
end-of-stage clearing of boxes strictly to the right of the current
approximation therefore never clears $\eta_x^\gamma$.  The clearing of an
$f$-box along an $\infty$-outcome also cannot clear it, since that operation
only applies to odd nodes which are initial segments of the current
approximation.
Finally, a clearing caused by an even proper initial segment of $\alpha$ would
initialize $\alpha$, again contrary to the hypothesis.  Therefore the box is
never cleared after stage $s$.
\end{proof}

\begin{lemma}\label{disposal-main}
Let $\alpha$ be the $Q_e$-node on the true path.  Let $s_0$ be a stage after
which the left of $\alpha$ is never visited and $\alpha$ is never initialized.
Suppose that $\alpha$ drops a claim $C=R(\rho,F)$ at an
$\alpha$-expansionary stage $s>s_0$.  Then for all $X \in C$, if $\Psi^X$ is total, then $(\exists n)\Psi^X\upharpoonright n\in W_e^A$.

\end{lemma}

\begin{proof}
Let $s$ be the stage at which the claim is dropped.  For every
$\tau\in2^{s-1}$ inside $C$, the construction chooses
$\nu_\tau\in W_e^A[s]$ extending $\Psi^\tau$ and defines
$\Psi^{\tau0}=\Psi^{\tau1}=\nu_\tau$.  Let $w$ be the maximum of the oracle
uses of these finitely many enumerations.
At the same step of the construction we set $r_\alpha=w^+=\max\{w,r_\alpha[s]\}$.  Hence every
$W_e^A$-enumeration used by $\alpha$ since its last initialization has oracle use
at most $r_\alpha$ after stage $s$.

It remains to show that no later enumeration into $A$ below $w^+$ can occur.  The
protection clause in the definition of $\alpha$-expansionary, applied with
$w^+$, ensures that at stage $s$ there is no dangerous box of use at most
$w^+$ above $\alpha$ along an $\infty$-outcome, nor to the right of $\alpha$.
Thus later clearings of boxes of these two kinds cannot change $A$ below
$w^+$: boxes already present there have use greater than $w^+$, and boxes filled
later have fresh use greater than $w^+$.
It remains to consider boxes attached to odd extensions of $\alpha$.  In
Subcase 1-1 there is no such extension box of use at most $w^+$. In Subcase 1-2, the only possible small extension boxes are boxes already
present at stage $s$.  By Lemma \ref{extension-jump-main}, these boxes are not
cleared later while the hypotheses of the present lemma remain in force. 
\end{proof}

\begin{lemma}\label{negative-main}
All negative requirements are satisfied.
\end{lemma}

\begin{proof}
Suppose, towards a contradiction, that some $Q_e$ fails.  Then $W_e^A$ is
dense, while for some $X$, $\Psi^X$ is total and $\Psi^X\upharpoonright n\notin W_e^A$ for all $n$.
Let $\alpha$ be the $Q_e$-node on the true path.  Fix a stage $s_0$ after
which $\alpha$ is never initialized and no node to its left is visited.  We
may take $s_0$ to be the stage of the last initialization of $\alpha$: any
later visit to a node to its left, or any later change in a higher-priority
claim relevant to $\alpha$, would initialize $\alpha$ again.  Thus all such
claims are fixed after $s_0$, and every base used by $\alpha$ after this
initialization is claimed at a stage greater than $s_0$.

We first show that $\alpha$ is eventually satisfied.  By Lemma
\ref{finite-injury-main}, choose $s_1>s_0$ after the last action of $\alpha$.
Suppose that $\alpha$ is not satisfied.  Since $\alpha$ is visited infinitely
often, it must then have a claim $C=R(\rho,F)$ which remains current after
$s_1$; otherwise its next visit would fall under Case~3 and make $\alpha$ act.
Choose $t\geq\max\{|\xi|:\xi\in F\cup\{\rho\}\}$.
For each $\tau\in2^t$ inside $C$, density gives a string in $W_e^A$ properly
extending $\Psi^\tau$.  Choose stable computations witnessing these finitely
many extensions, let $u$ bound their oracle uses, and put
$u^+=\max\{u,r_\alpha\}$.  While $C$ remains current, clause~(v) keeps $\Psi$
constant above each such $\tau$ inside $C$.  By the least-use convention in
Definition~\ref{d6}(iii), at every sufficiently late stage the witnesses
chosen there therefore have use at most $u$.
Choose the first visit to $\alpha$ at a stage $s$ so late that the stable
computations and the number $u^+$ have appeared, $s-1\geq t$, and the bound in
the preceding paragraph applies.  At this visit the stage-specific value of
$w^+$ is at most $u^+$.  If the visit is expansionary, then Case~1 already makes
$\alpha$ act.  Otherwise only the protection condition can fail, and the
end-of-stage actions clear all boxes relevant to Definition~\ref{d6}(iv) whose
uses are at most $u^+$.  Any such box filled afterwards has fresh use greater
than $u^+$.  Hence the next visit to $\alpha$ is expansionary, and Case~1 again
makes $\alpha$ act.  Either case contradicts the choice of $s_1$.  Thus
$\alpha$ is eventually satisfied.

Now suppose that the base of a fixed higher-priority claim is an initial
segment of $X$, and choose such a base $\eta$ of maximal length; there are only
finitely many fixed claims.  If the corresponding claim is $R(\eta,F)$, then
no member of $F$ is an initial segment of $X$, since its recorded exclusions
are bases of fixed higher-priority claims and any such member would be longer
than $\eta$.
Consequently $X\in R(\eta,F)$.  This claim remains current, so clause~(v)
keeps $\Psi^{X\upharpoonright m}$ constant for all sufficiently large $m$,
contrary to the totality of $\Psi^X$.  Therefore no fixed higher-priority base
is an initial segment of $X$.
Let $\rho=X\upharpoonright k_\alpha$.  The preceding conclusion implies that
$\rho$ is legal for $\alpha$.  Indeed, no fixed higher-priority base is an
initial segment of $\rho$, and the fixed bases extending $\rho$ cannot cover
$[\rho]$, since $X$ belongs to none of their cones.  Moreover, $X$ belongs to
the corresponding clopen remainder.  Since the higher-priority claims are
fixed, $\rho$ remains legal after $s_0$.  After its last initialization,
$\alpha$ uses each legal base at most once and is declared satisfied only when
no legal unused base remains.  Thus $\rho$ was claimed and its claim was later
dropped after $s_0$.  Lemma~\ref{disposal-main} yields
$(\exists n)\,\Psi^X\upharpoonright n\in W_e^A$, contrary to the choice of $X$.
\end{proof}

\begin{lemma}\label{15}
The global requirement $M$ is satisfied.
\end{lemma}

\begin{proof}
For each stage $s$, let $C_s$ be the union of the clopen sets currently claimed
by even nodes at the end of stage $s$.  If $X\notin C_s$ for some arbitrarily large
stages $s$, then clause (v) appends a bit to $\Psi$ along $X$ at arbitrarily large
stages, and hence $\Psi^X$ is total.  Therefore the set of oracles $X$ such that $\Psi^X$ is not total is a subset of $\liminf_s C_s$.
By Lemma \ref{claims-antichain-main}, the current claimed clopen sets are
pairwise disjoint.  If $\alpha$ has base length $k_\alpha$, then its claimed
clopen set is contained in a cone of measure $2^{-k_\alpha}$.  The parameters
$k_\alpha$ are pairwise distinct and all at least $2$, so for every $s$, $\mu(C_s)\leq\sum_{n\geq2}2^{-n}=\frac12$.
Fatou's lemma for measure gives $\mu(\liminf_s C_s)
      \leq \liminf_s \mu(C_s)
      \leq \frac12$.
Thus $\Psi$ is total on a class of measure at least $1/2$.
\end{proof}

\begin{lemma}\label{pos}
All positive requirements are satisfied.
\end{lemma}

\begin{proof}
Fix $e$.  We verify that the truth table $\sigma_e$ satisfies $e\in\mathrm{Tot}\Longleftrightarrow A'\vDash\sigma_e$.
Let $\alpha$ be the node of length $2e+1$ on the true path.  Choose a stage
$s_0$ after which no node to the left of $\alpha$ is visited and no even proper
initial segment of $\alpha$ acts.  From then on, every node of length $2e+1$
strictly to the right of $\alpha$ has both boxes cleared whenever $\alpha$ is
visited.  Boxes to the left of $\alpha$ may remain occupied, but the truth table
reads the level from right to left, so they are irrelevant once one of the boxes
of $\alpha$ is permanent.

Suppose first that $\Phi_e$ is total.  Then $\alpha$ has infinitely many
expansionary stages.  Choose an $\alpha$-expansionary stage $s_1>s_0$.  At the
end of $s_1$, the box $\eta_f^\alpha$ has been cleared and
$\eta_\infty^\alpha$ is filled if it was empty.  Let $u$ be the use of
$\eta_\infty^\alpha$.  We check that no later enumeration into $A$ below $u$ can
destroy this computation.  No node strictly to the left of $\alpha$ is visited
after $s_0$, and no even proper initial segment of $\alpha$ acts after $s_0$.
Every box to the right of $\alpha$, and every $f$-box on an odd predecessor
whose $\infty$-outcome leads to $\alpha$, has just been cleared or will be
refilled only with fresh use greater than $u$.  Every box properly below
$\alpha\infty$ which is nonempty after $s_1$ has use greater than $u$ by Lemma
\ref{box-nesting-main}, and boxes filled later also have fresh use greater than
$u$.  Therefore later clearings can enumerate only numbers greater than $u$.
Hence $\eta_\infty^\alpha$ is permanently occupied, while $\eta_f^\alpha$ is
cleared at later $\alpha$-expansionary stages and is not permanent.  Therefore
$A'\vDash\sigma_e$.

Suppose now that $\Phi_e$ is not total.  Then $\alpha$ has only finitely many
expansionary stages.  Choose $s_1>s_0$ after the last such stage and then a
later visit $s_2$ to $\alpha$ with outcome $f$, after all boxes to the right of
$\alpha$ filled before $s_2$ have been cleared.  At the end of $s_2$, the box
$\eta_f^\alpha$ is filled if it is empty; let $u$ be its use.  There is no later
expansionary visit to $\alpha$, so the construction never again clears
$\eta_f^\alpha$ because of an $\infty$-outcome at $\alpha$.  Also, after $s_0$
there is no left visit and no action by an even proper initial segment of
$\alpha$.  The boxes to the right of $\alpha$ which existed before $s_2$ have
already been cleared, any box below the later path has use greater than $u$ by
Lemma \ref{box-nesting-main}, and every box filled after $s_2$ receives fresh
use greater than $u$.  Thus every later clearing which can occur has use greater
than $u$, and $\eta_f^\alpha$ is permanently occupied.  Since the truth table
reads $\eta_f^\alpha$ before $\eta_\infty^\alpha$, we have
$A'\nvDash\sigma_e$.
\end{proof}

\begin{proof}[Proof of Theorem \ref{main}]
By Lemma \ref{pos}, $A$ is superhigh.  By Lemma \ref{negative-main}, \ref{15} and the discussion in Subsection \ref{requ}, $A$ is AEWG.
\end{proof}

\section{Proof of Theorem \ref{low2}}

Fix a $\text{low}_2$ r.e.
set $A$ and a recursive increasing enumeration $A=\bigcup_s A_s$. Since $A$ is given to us, we have only 
the requirements $Q_e$ and $M$ from Subsection \ref{requ}.

In the previous section, when we properly extend the definition of $\Psi$, the negative
requirement imposes its restraint by clear boxes in order to keep the
enumerations in $W_e^A$. Now we need to consider the case that some
number less than the use of enumerations will be enumerated into $A$ later which we cannot control.
To coexist with that, we use the ``certification process'' developed recently
by Cholak, Downey and Greenberg {\cite{CDG}}. The authors used this technique together with the $\Delta_3$ guessing argument
to prove that every $\text{low}_2$ r.e. set has an atomless hyperhypersimple superset, which is a milestone towards the Soare's $\text{low}_2$ Conjecture. This result is irrelevant but the technique in their proof is significant for this section. 

The proof that every low r.e. set has the AEWG property in \cite{MR4741732} uses Robinson's trick for guessing. This is done by asking $\Sigma^{0}_1(A)$ questions to $\emptyset'$, and using the approximation solutions to the questions given by the Limit Lemma and the Recursion Theorem to decide what to do in the construction. We do not use guessing arguments here though there is a version of guessing called the $\Delta_3$ guessing which is suitable for low$_2$ r.e. sets (see \cite{CDG} for example). Instead, the certification process uses domination property of low$_2$ sets and there is no need for the Recursion Theorem. We will explain how we use the ``certification process'' after introducing the setup.

We first record the consequence of the relativized High Domination Theorem
that will be used in the certification process.

\begin{lemma}\label{low2:dominating-approximation}
There is a $\emptyset'$-recursive function $f$ which dominates every
total $A$-recursive function. Moreover, $f$ has a recursive approximation
$\{f_s\}_{s\in\omega}$ such that $f_s(x)\leq f_{s+1}(x)$
for every $x$.
\end{lemma}

\begin{proof}
Since $A$ is r.e., $A\leq_T\emptyset'$. Since it is low$_2$, $A''\equiv_T\emptyset''=(\emptyset')'$.
Hence $\emptyset'$ is high relative to $A$. The relativized High Domination
Theorem gives an $\emptyset'$-recursive function $h$ which dominates every
total $A$-recursive function. Let $\{h_s\}$ be a convergent recursive
approximation to $h$ and put $f_s(x)=\max_{t\leq s}h_t(x)$,
$f(x)=\lim_s f_s(x)$. For each fixed $x$ the limit exists. Also $f\geq h$, so $f$ still has the
required domination property, and its displayed approximation is monotone.
\end{proof}

The monotonicity of $f_s(x)$ will imply that, at each fixed coordinate, only
finitely many distinct attempt records ever become certified.

\subsection{The priority tree and objects}

We first define all of the finite objects used by the stage construction. As numerous concepts are defined in this section, every concept will be typeset \textbf{in bold} at its first appearance.

The tree of strategies is $T=\omega^{<\omega}$. A node $\alpha$ of length
$e$ runs the module for $Q_e$, and its children are
$\beta_i=\alpha i$ ($i\in\omega$), which are workers for $Q_e$. As in Section 2, write
$\gamma<_{\mathrm{left}}\beta$ when $\gamma$ lies lexicographically to the
left of $\beta$, and put $\gamma<_L\beta$ if
$\gamma\subsetneq\beta$ or $\gamma<_{\mathrm{left}}\beta$.

Fix a recursive one-to-one parameter function $k:T\setminus\{\lambda\}\longrightarrow\{2,3,4,\ldots\}$, $k(\beta)=k_\beta$,
such that $k_{\alpha i}<k_{\alpha j}$ whenever $i<j$. We call $k_\beta$
the \textbf{parameter} of $\beta$.

As in Subsection \ref{2.2}, a \textbf{claim} for a node $\beta$ is a pair
$(\rho,F)$, where $\rho\in2^{k_\beta}$, $F$ is a finite prefix-free set of
proper extensions of $\rho$, and
$R(\rho,F)=[\rho]\setminus\bigcup_{\xi\in F}[\xi]\neq\varnothing$.
We call $\rho$ the \textbf{base} of the claim and $R(\rho,F)$ its
\textbf{claimed clopen set}. At any stage a nonempty node has at most one
current claim. If the current claim of $\gamma$ is $(\rho,F)$, put
$C_s(\gamma)=R(\rho,F)$; if $\gamma$ has no current claim, put
$C_s(\gamma)=\varnothing$. For a worker $\beta$, define its
\textbf{higher-priority claim set} by
$\mathcal H_s(\beta)=\bigcup_{\gamma<_L\beta}C_s(\gamma)$.
For $\rho\in2^{k_\beta}$, the \textbf{current residual} of the base $\rho$
for $\beta$ is $R_s(\beta,\rho)=[\rho]\setminus\mathcal H_s(\beta)$.
When this residual is nonempty, let $F_s(\beta,\rho)$ be the finite
prefix-minimal frontier for $\mathcal H_s(\beta)\cap[\rho]$. Then every member
of $F_s(\beta,\rho)$ is an extension of $\rho$ and $R_s(\beta,\rho)=R(\rho,F_s(\beta,\rho))$.
We retain the notation $R_s(\beta,\rho)$ when the residual is empty. To create
a claim from the base $\rho$ means to record the pair
$(\rho,F_s(\beta,\rho))$; its claimed clopen set is exactly
$R_s(\beta,\rho)$.

For each parent $\alpha$ and \textbf{coordinate} $x$, the construction recursively
enumerates a sequence of \textbf{attempt records}
\[
 r_{\alpha,x,n}=
 (\sigma_{\alpha,x,n},v_{\alpha,x,n},b_{\alpha,x,n},
   \mathcal E_{\alpha,x,n}).
\]
Here $\sigma_{\alpha,x,n}\in2^{b_{\alpha,x,n}}$ is the \textbf{guard},
$v_{\alpha,x,n}$ is the \textbf{value}, and $\mathcal E_{\alpha,x,n}$ is the
\textbf{witness package}. The package records the worker $\beta$, its exact
current claim $(\rho,F)$, and the finite
$W_e^A[s]$-computations selected when the record is issued. The records for
$(\alpha,x)$ are numbered consecutively. Record $n+1$ is issued only after
the current approximation to $A$ is incompatible with the guard of every
record of index at most $n$. Issued attempt records are permanent entries in
the ordered-search program: no later rule deletes them.

For every oracle $Z$, define the \textbf{ordered search} functional by $\Theta^Z(\alpha,x)=v_{\alpha,x,n}$,
where $n$ is least such that $\sigma_{\alpha,x,n}\subset Z$; if there is no
such record, the computation diverges. The properties of this search are
proved in Lemma \ref{low2:ordered-search-facts} below.

A record issued at stage $t$ with guard $\sigma=A_t\upharpoonright b$ is
\textbf{live} at stage $s\geq t$ if $A_s\upharpoonright b=\sigma$. It is
\textbf{killed} when this equality first fails. The \textbf{least gap} at
$\alpha$ at stage $s$ is
\[
        \operatorname{gap}_\alpha(s)=
        \min\{x:\text{there is no live record for }(\alpha,x) \text{ at stage }s\}.
\]

Suppose that the stored witness package has maximum oracle use $w$. If $x>0$,
let $b_{x-1}$ be the guard length of the current live record at $x-1$. A new
record at the least gap $x$ is issued according to the \textbf{nested-guard
rule}
\begin{equation}\label{low2:nested-guard}
 b=
 \begin{cases}
     w+1, & x=0,\\
     \max\{w,b_{x-1}\}+1, & x>0,
 \end{cases}
 \qquad
 \sigma=A_s\upharpoonright b,
 \qquad
 v=f_s(x)+1.
\end{equation}
The number $b$ need not exceed the guard lengths of previously killed records
at coordinate $x$. This bounded-retry feature is essential in the density
argument.

An \textbf{assignment} associates an attempt record with one child and is
created only while that record is live. If an assigned record is killed, its
assignment link is retained provisionally until the pending injury is
processed, unless an intervening initialization unassigns it. Since
assignments are created only for live records, once the link of a killed
record has been removed, it can never be created again. At any stage, at most
one record is assigned to a child and a record is assigned to at most one
child. We write
$\operatorname{coord}_s(\beta)=x$ when the record at coordinate $x$ is
assigned to $\beta$, and write $\operatorname{coord}_s(\beta)\uparrow$ when
no record is assigned to $\beta$.

Suppose that a record assigned to $\beta$ has coordinate $x$, was issued at
stage $t$, and has value $v=f_t(x)+1$. It is \textbf{certified} at stage $s$
if $f_s(x)>v$.

Each attempt record carries a Boolean \textbf{success mark}, initially off.
Case 1 below is the only action which turns this mark on. We call a record
\textbf{successful} while its success mark is on. A successful record is
\textbf{valid for $\beta$ at stage $s$} if it records $\beta$ as its worker,
is live at $s$, and the claimed clopen set stored in its witness package is
exactly $R_s(\beta,\rho)$ for its stored base $\rho$.
Thus all finite evidence of a successful certification is stored in the
attempt record itself.

We say that a parent $\alpha$ is \textbf{visited} when its local module is
processed during the stage traversal specified below; a visit to $\alpha$
means such an execution of its module.

After pending injuries (to be defined) have been processed, a worker $\beta$ is
\textbf{satisfied} at a visit if it has no current claim and, for every
$\rho\in2^{k_\beta}$, either $R_s(\beta,\rho)=\varnothing$ or there is a
record valid for $\beta$ which stores exactly this residual. We call $\beta$
\textbf{eventually satisfied} if it is satisfied at every sufficiently late
visit to its parent, after pending injuries have been processed.

The construction uses a one-stage \textbf{coordinate cooldown} on a record
coordinate and a finite \textbf{pending $\alpha$-injury} for each parent
$\alpha$. The latter records
the finite set of attempt records in the $\alpha$-module killed since the
previous visit, the coordinates killed with them by the nested-guard rule,
and the assignments carried by those records when they were killed. Their
operational effects are specified in the construction below. 

A pair $(\rho,F)$ is
\textbf{resolved at stage $s$} if
$s\geq\max\{|\xi|:\xi\in F\cup\{\rho\}\}$.

At a visit to a parent $\alpha$, after its pending injury has been processed,
the \textbf{frontier} is the first child having no assigned record. If the
frontier $\beta_i$ has current claim $(\rho,F)$ with claimed clopen set
$C=R(\rho,F)$, then this claim is \textbf{expansionary at stage $s$} if, for
every $\tau\in2^{s-1}$ with $[\tau]\subseteq C$, there is a string $\nu_\tau\in W_e^A[s]$ properly extending $\Psi^\tau$, where $e=|\alpha|$.
We choose witnesses having least oracle use and let $w$ be the maximum of
these uses.

To \textbf{initialize} a node $\gamma$ means to remove its current claim,
unassign any record assigned to it, turn off the success mark of every attempt
record whose stored worker is $\gamma$, clear every coordinate cooldown in the
$\gamma$-module, and delete any pending $\gamma$-injury. To initialize a
subtree means to initialize every node in that subtree.

\subsection{The construction}

At stage $0$, put $\Psi^\lambda=\lambda$. There are no claims, records,
assignments, or active success marks, and all cooldowns and pending injuries
are empty.

Whenever a claim for a node $\eta$ is created or removed, apply the
\textbf{claim-change initialization rule}: initialize every $\gamma$ with
$\eta<_L\gamma$. (Lemma \ref{low2:claim-separation} below proves that the
claimed clopen sets produced by this rule are pairwise disjoint and have total
measure at most $1/2$.)

At every finite stage, only finitely many nodes carry state on which
initialization acts. Thus an instruction to initialize a collection of nodes
means to apply the preceding initialization operations to its finitely many
relevant members; all other members of the collection are already in their
default state.

A general stage $s>0$ proceeds as follows. At the beginning of the stage,
before the traversal, inspect the finitely many records issued before stage
$s$. Every record which was live at stage $s-1$ and is not live at stage $s$
is killed and is entered, together with its coordinate and its assignment at
that moment, into the appropriate pending injury. The stage traversal then
begins at the root. Whenever it reaches a parent $\alpha$, first process any
pending $\alpha$-injury and determine the resulting frontier. Then apply the
first applicable local case; that case carries out the local action and
specifies the outcome of $\alpha$. The traversal follows this outcome to the
corresponding child and repeats the same procedure, for at most $s$ levels.
After the traversal ends, define the stage-$s$ output table for $\Psi$. The
following three subsections specify these steps.

\subsubsection{Injury processing}

When a change in $A$ kills some live records in the $\alpha$-module, add
those killed records, together with their coordinates and assignment data, to
the pending $\alpha$-injury; if further deaths occur before the next visit,
their data are added to the same finite pending injury. It is processed at the
next visit to $\alpha$, before any other local action. If the module is never
visited again, its pending injury is never processed, and the module performs
no further action.

When processing this injury, turn off the success marks of all killed records.
Among the killed records, consider only those whose assignment links are
still present when the injury is processed. If there is no such record,
initialize no child. Otherwise, let $\beta_j=\alpha j$ be the least-indexed
child to which one of these killed records is still assigned. Unassign every
record currently assigned to $\beta_j,\beta_{j+1},\ldots$, leave the current
claim of $\beta_j$ unchanged, and initialize every $\beta_\ell$ with
$\ell>j$, together with its subtree. The proper extensions of $\beta_j$ are
not initialized, because its claim has not changed and hence neither have
their higher-priority claim sets. In all cases, put every coordinate recorded
in the pending injury---not only the coordinates of records whose assignment
links are still present---in one-stage cooldown. Delete the pending injury
after these actions have been performed. Clear the resulting cooldowns at the
end of the stage.

If a least affected child $\beta_j$ exists, it is the resulting frontier and
its retained current claim is nonempty, so Case 3 cannot apply at this visit.
Moreover, the current least gap is among the coordinates just put in cooldown,
so Case 4 cannot apply either.

\subsubsection{The local cases}

Suppose $\alpha$ is visited at stage $s$. First process any pending
$\alpha$-injury and let $\beta_i=\alpha i$ be the resulting frontier. (Lemma
\ref{low2:frontier} below proves that precisely the children
$\beta_0,\ldots,\beta_{i-1}$ are assigned and that their coordinates are
strictly increasing.) Apply the first case whose hypotheses hold; it specifies
the outcome of $\alpha$.

\medskip
\noindent\textbf{Case 1.}
If the record assigned to some $\beta_j$, $j<i$, is certified, select the
least such $j$. Use its witness package to extend $\Psi$ throughout the
claimed clopen set of its stored claim, and turn on its success mark. Then
remove that claim and unassign the record. Apply the claim-change
initialization rule and set the outcome of $\alpha$ to $j$. (The legitimacy of
the stored computations is proved in Lemma
\ref{low2:certification-safety} below.)

\medskip
\noindent\textbf{Case 2.}
If the frontier $\beta_i$ is satisfied, make no state change and set the
outcome of $\alpha$ to $i$.

\medskip
\noindent\textbf{Case 3.}
Suppose $\beta_i$ has no current claim and the current stage resolves every
nonempty pair
$(\rho,F_s(\beta_i,\rho))$, $\rho\in2^{k_{\beta_i}}$. Choose the least
$\rho$ whose current residual is nonempty and for which there is no valid
record storing that exact residual. Create the claim
$(\rho,F_s(\beta_i,\rho))$, apply the claim-change initialization rule, and
set the outcome of $\alpha$ to $i$.

\medskip
\noindent\textbf{Case 4.}
If the frontier has a current expansionary claim and no assignment, and if
$x=\operatorname{gap}_\alpha(s)$ is not in
coordinate cooldown, issue an attempt record at $x$ according to
\eqref{low2:nested-guard}, assign it to $\beta_i$, and set the outcome of
$\alpha$ to $i$.

\medskip
\noindent\textbf{Case 5.}
In every remaining case, make no state change and set the outcome of $\alpha$
to $i$.

\medskip
Since the first applicable case is used, exactly one case is executed at a
visit. Only the frontier is tested for satisfaction; no child to its right is
processed by the local module.

\subsubsection{Stage traversal and the functional $\Psi$}

At stage $s$, start at the root and follow the outcomes just defined for at
most $s$ levels. The resulting finite sequence $\delta_s$ is the
\textbf{stage path}; equivalently, its initial segments are exactly the nodes
visited at stage $s$.

We define the output tables recursively at the end of stage $s\geq1$. Fix
$\tau\in2^s$, let $\bar\tau=\tau\upharpoonright(s-1)$, and put
$q=\Psi^{\bar\tau}$. Define $\Psi^\tau$ as follows.
\begin{enumerate}[label=(\roman*)]
\item Suppose $[\tau]$ lies inside the claimed clopen set of a claim removed
      by Case 1 at stage $s$, and the corresponding attempt record was issued
      at stage $u$. Let $\eta=\tau\upharpoonright(u-1)$. Using the unique
      stored witness $\nu_\eta$, put $\Psi^\tau=\nu_\eta$.
\item Otherwise, if $[\tau]$ lies inside the claimed clopen set of a current
      claim, put $\Psi^\tau=q$.
\item In every other case, put $\Psi^\tau=q\,^{\smallfrown}0$.
\end{enumerate}

\subsection{Verification}

We first prove the elementary properties of the construction. 

\begin{lemma}\label{low2:claim-separation}
At every point of the construction after the initialization required by the
current action has been performed, the claimed clopen sets of distinct
current claims are pairwise disjoint. Moreover, the union of all current
claimed clopen sets has measure at most $1/2$.
\end{lemma}

\begin{proof}
When a claim for $\beta$ is created, its claimed clopen set is a current
residual and is therefore disjoint from the claimed clopen set of every
surviving node $\gamma<_L\beta$. The claim-change initialization rule removes
the claims of every node $\gamma$ with $\beta<_L\gamma$. Removing further
claims cannot create an intersection, so the claimed clopen sets remain
pairwise disjoint.
The claimed clopen set of a $\beta$-claim is contained in a cone of measure
$2^{-k_\beta}$. Since $\beta\mapsto k_\beta$ is one-to-one and
$k_\beta\geq2$, the union of all current claimed clopen sets has measure at
most $\sum_{n\geq2}2^{-n}=\frac12$.
\end{proof}

\begin{lemma}\label{low2:ordered-search-facts}
For every node $\alpha$ the following hold.
\begin{enumerate}[label=(\roman*)]
\item $\Theta^A(\alpha,\cdot)$ is an $A$-partial recursive function.
\item For every oracle $Z$, the domain of $\Theta^Z(\alpha,\cdot)$ is an
      initial segment of $\omega$.
\item If $N\notin\operatorname{dom}(\Theta^A(\alpha,\cdot))$, then stages at
      which coordinate $N$ has no live record occur cofinally.
\item Killing a live record at $x$ kills every currently live record at a
      coordinate greater than $x$.
\item At every stage, the coordinates carrying live records form an initial
      segment of $\omega$.
\end{enumerate}
\end{lemma}

\begin{proof}
First note that there is at most one live record at each coordinate. A new
record at $(\alpha,x)$ is issued only after the current approximation to $A$
is incompatible with the guards of all earlier records there, and the
monotonicity of the r.e. approximation prevents any killed record from
becoming live again.
So part (i) follows from the ordered search. For (ii), suppose that a record at
$x>0$ has a guard compatible with $Z$. At its issue stage, its guard strictly
extended the guard of the live record at $x-1$. That lower guard is also
compatible with $Z$, so some record at $x-1$ is accepted by the ordered
search. Induction gives downward closure.

If $N\notin\operatorname{dom}(\Theta^A(\alpha,\cdot))$, every record at $N$
has a guard incompatible with $A$ and is eventually killed. After each such
record is killed, either $\alpha$ is never visited again, in which case no
later $N$-record is issued, or its pending injury survives until the next
visit, when the coordinate cooldown prevents the issue of an $N$-record. The
only remaining possibility is that the pending injury is deleted because
$\alpha$ is initialized before its next visit. Every construction step that
does this also initializes all children of $\alpha$, either by the
claim-change initialization rule or by initializing a subtree. Hence, at the
first later visit to $\alpha$, no frontier child has a current claim, so Case
4 cannot issue an $N$-record at that visit. Thus stages with no live
$N$-record occur cofinally, proving (iii).
Part (iv) follows directly from the same strict extension relation. For (v),
records are issued only at the current least gap, while the death of a live
record kills every live record at a higher coordinate by (iv). Certification
and unassignment do not change which records are live. These observations
preserve the asserted initial-segment property at every stage.
\end{proof}

\begin{lemma}\label{low2:resolution}
Suppose a claim $(\rho,F)$, with claimed clopen set $C=R(\rho,F)$, is created
at stage $t$. Then $(\rho,F)$ is resolved at stage $t$. If an attempt record
for this claim is issued at stage $u$, then $u>t$, and the level-$(u-1)$ cones
contained in $C$ form a finite partition of $C$.
\end{lemma}

\begin{proof}
Case 3 creates $(\rho,F)$ only after the current stage resolves it. At most
one local case is applied at a visit, so Case 4 cannot issue a record for the
new claim until a later visit, at a stage $u>t$. Hence $u-1\geq t$, and every
level-$(u-1)$ cone is either contained in or disjoint from $C$. The cones
contained in $C$ therefore form the asserted finite partition.
\end{proof}

\begin{lemma}\label{low2:frontier}
Immediately after any pending $\alpha$-injury, if one exists, has been
processed at a visit to $\alpha$, there is a unique $i$ such that precisely
$\beta_0,\ldots,\beta_{i-1}$ are assigned records. Their record coordinates
are strictly increasing, every assigned record is live, and every assigned
record stores its assignee as its worker and stores exactly that child's
current claim. In particular, every assigned child has a nonempty current
claim. Every child $\beta_j$ with $j>i$ is in its initialized state and is not
satisfied.
\end{lemma}

\begin{proof}
The statement holds initially. The death of a record does not by itself remove
its assignment link. Before the pending injury is processed, such a
provisional link can disappear only through initialization, since no further
local action in the $\alpha$-module occurs before injury processing and a
killed record cannot be assigned again. Every such initialization either
initializes $\alpha$ itself, in which case the pending $\alpha$-injury is
deleted, or removes an assigned suffix of the children of $\alpha$ and
initializes that suffix. Thus a killed assignment whose link is no longer
present requires no further local initialization.

Suppose now that some killed assignment link is still present, and let
$\beta_j$ be the least child carrying one. By the inductive strict increase of
the assigned coordinates and Lemma
\ref{low2:ordered-search-facts}(iv), every record still assigned to a child to
the right of $\beta_j$ is killed as well. Injury processing therefore removes
a suffix of the current assignments, initializes every child strictly to the
right of the new frontier, and leaves only live assigned records. If no killed
assignment link remains, either no assigned record was killed or the affected
suffix has already been removed and initialized. The claim-change
initialization rule likewise initializes a suffix of the children in any fixed
local module.

Cases 2, 3, and 5 make no assignment. In Case 4, Lemma
\ref{low2:ordered-search-facts}(v) shows that the current least gap is greater
than every assigned coordinate. Assigning that coordinate to the frontier
therefore preserves both the initial-segment form and the strict increase of
the assigned coordinates. Case 1 unassigns one child and initializes every
child to its right, again leaving an initial segment.

An assignment is created only in Case 4. At that action the attempt record
stores the frontier as its worker and stores exactly the frontier's nonempty
current claim. Case 1 removes that claim and unassigns the record
simultaneously, and initialization also unassigns every affected record.
Injury processing leaves the current claim of the least affected child
unchanged while removing its assignment, and initializes every affected child
to its right. Thus every remaining assigned record stores its assignee and
exactly that child's nonempty current claim. The same induction shows that every child
strictly to the right of the frontier is initialized: the local module
processes only the frontier, and whenever the frontier moves left, it
initializes the entire suffix to its right. Such a child has no claim, no
assignment, and no active success mark. By Lemma
\ref{low2:claim-separation}, its higher-priority claim set has measure at most
$1/2$, so some base has a nonempty residual. Hence the child is not satisfied.
\end{proof}

\begin{lemma}\label{low2:certification-safety}
If an assigned attempt record is live at stage $s$, every computation in its
stored witness package is still valid in $W_e^A[s]$. If its guard is
compatible with the final set $A$, all of the stored strings belong to
$W_e^A$.
\end{lemma}

\begin{proof}
The guard length is strictly greater than every oracle use in the witness
package. Liveness gives $A_s\upharpoonright b=\sigma$, so all computations in
the package are still valid at stage $s$. If $\sigma\subset A$, the same
computations are permanent and give membership in $W_e^A$.
\end{proof}

\begin{lemma}\label{low2:finite-certification}
For fixed $\alpha$ and $x$, only finitely many distinct attempt records at
coordinate $x$ ever become certified.
\end{lemma}

\begin{proof}
Choose $S$ such that $f_s(x)=f(x)$ for every $s\geq S$. Every record at
coordinate $x$ issued at or after stage $S$ has value
$f_s(x)+1=f(x)+1$, and therefore can never satisfy the certification
inequality. Only finitely many records at this
coordinate can have been issued before stage $S$. Hence only finitely many
distinct records at coordinate $x$ ever become certified.
\end{proof}

\begin{lemma}\label{low2:functional}
The clauses defining $\Psi^\tau$ are unambiguous and monotone and are
uniformly recursive. Hence
$\Psi^X=\bigcup_s\Psi^{X\upharpoonright s}$ defines a recursive Turing
functional. Moreover, while a claim remains current, $\Psi$ is frozen on its
claimed clopen set.
\end{lemma}

\begin{proof}
By Lemmas \ref{low2:claim-separation} and \ref{low2:resolution}, every
level-$s$ cone is either contained in or disjoint from every claimed clopen
set relevant at stage $s$. The claimed clopen sets removed by Case 1 at the
same stage are pairwise disjoint, and the claim-change initialization rule
removes the appropriate lower-priority state before traversal continues.
Thus clause (i) has a unique witness package.

Suppose a claim $(\rho,F)$ with claimed clopen set $C$ is created at stage
$t$ and removed by Case 1 at stage $s>t$. Clause (ii) freezes the output on
$C$ from the end of stage $t$ through stage $s-1$. If its attempt record was
issued at stage $u$, then, for every relevant $\tau$,
\[
        \Psi^{\tau\upharpoonright(s-1)}
        =\Psi^{\tau\upharpoonright(u-1)}
        \subsetneq \nu_{\tau\upharpoonright(u-1)}.
\]
Hence clause (i) extends the previous value and the recursion is monotone. A
claim created inside a region removed by Case 1 earlier in the same stage is
subject to clause (i) at that stage; its frozen value begins at the end of its
creation stage, and no record for it is issued until a later visit. All finite
stage data are recursive, so the tables are uniformly recursive. The final
assertions follow.
\end{proof}

The existence of the true path and the verification of the requirements must
be proved simultaneously. The following stability notion records exactly the
induction hypothesis needed at one parent.

\begin{definition}\label{low2:stable-true-parent}
A node $\alpha$ is a \textbf{stable true parent} if it is visited infinitely
often and, after some stage, the following hold.
\begin{enumerate}[label=(\alph*)]
\item No node lexicographically to the left of $\alpha$ is visited, and
      $\alpha$ is never initialized again by any construction step or injury
      processing.
\item All claims belonging to nodes $\gamma<_L\alpha$ are fixed.
\item Unless $\alpha=\lambda$, its claim status as a child in its parent's
      module is fixed in the following sense: either $\alpha$ is permanently
      satisfied or it has one fixed current claim. In the latter case,
      attempt records may still be assigned to and unassigned from $\alpha$
      without changing that claim.
\end{enumerate}
\end{definition}

A child is \textbf{settled} if, from some stage onward, one fixed live record
is assigned to it. Such a record is compatible with $A$.

\begin{lemma}\label{low2:not-all-settle}
Let $\alpha$ be a stable true parent. If no child of $\alpha$ is eventually
satisfied, then not every child is settled.
\end{lemma}

\begin{proof}
Suppose every child $\beta_i$ is eventually assigned a fixed record with
coordinate $x_i$ and value $v_i$.
The coordinates are distinct and hence unbounded. The compatible records at
arbitrarily large coordinates, together with Lemma
\ref{low2:ordered-search-facts}(ii), make
$\Theta^A(\alpha,\cdot)$ total. By Lemma
\ref{low2:dominating-approximation}, choose $i$ so large that $f(x_i)>\Theta^A(\alpha,x_i)=v_i$.
Eventually $f_s(x_i)>v_i$. At the next visit to $\alpha$, Case 1 selects the
least child among $\beta_0,\ldots,\beta_i$ whose assigned record is certified;
Case 1 precedes the satisfaction test. It unassigns that child's fixed record,
a contradiction.
\end{proof}

Let $i$ be the least index such that $\beta_i$ is not settled. Every
$\beta_j$ with $j<i$ is settled. By Lemma \ref{low2:frontier}, the assigned
children form an initial segment, so outcomes $j<i$ eventually cease. Their
subtrees receive no later
visits, their claims remain fixed, and no later construction step to the left
initializes $\beta_i$. Thus the higher-priority claim set of $\beta_i$ is
fixed from some stage onward.

\begin{lemma}\label{low2:auxiliary-partial}
Assume that $\alpha$ is a stable true parent, no child of $\alpha$ is
eventually satisfied, and $i$ is the least index of a nonsettled child. Then
$\Theta^A(\alpha,\cdot)$ is not total.
\end{lemma}

\begin{proof}
Suppose it is total. Final $A$-compatible records at arbitrarily large
coordinates are issued arbitrarily late. After the stabilization of the
children to the left of $i$, every such sufficiently late record is issued at
a visit whose frontier is some $\beta_j$ with $j\geq i$.

If $j=i$, the record assigned to $\beta_i$ is compatible with $A$. If $j>i$,
Lemma \ref{low2:frontier} says that a lower-coordinate record is assigned to
$\beta_i$ at that moment, and the nested-guard rule implies that this record
is compatible with $A$ as well. It cannot be killed, and no later
initialization from the left can unassign it. If it were never certified,
$\beta_i$ would be settled, contrary to the choice of $i$. It is therefore
eventually certified. Case 1 processes it and makes it successful.
Because the record is compatible with $A$ and the higher-priority claim set
of $\beta_i$ is fixed, and because $\beta_i$ is never initialized again, its
success mark remains on and the record remains valid for $\beta_i$
permanently.

Case 1 must make infinitely many such permanent records successful for
$\beta_i$. Otherwise, fix a stage after the last one is processed and after
all the stabilization described above. By totality, a final compatible record
is issued later at a visit whose frontier is some $\beta_j$ with $j\geq i$.
At its issue stage, a compatible record is assigned to $\beta_i$: it is the
newly issued record if $j=i$, and it is a lower-coordinate record if $j>i$.
This record cannot be killed, and no initialization from the left can unassign
it. Hence it must eventually be certified; otherwise $\beta_i$ would be
settled. Case 1 then makes another permanent record successful, a
contradiction. A temporarily satisfied state or a permanently
nonexpansionary claim cannot occur throughout this argument: in the first
case $\beta_i$ would be eventually satisfied, and in the second case only
finitely many coordinates of the total ordered search could receive records.

Since the higher-priority claim set of $\beta_i$ is fixed, each permanently
valid successful record produced in this way corresponds to a base for which
there was previously no valid record. After this stabilization, the residual
at each base is fixed, so a permanently valid record for that residual
prevents Case 3 from ever selecting the same base again. The records therefore
correspond to distinct bases. There are only $2^{k_{\beta_i}}$ bases,
contradicting the existence of infinitely many such records.
\end{proof}

Put
$N=\min\bigl(\omega\setminus
           \operatorname{dom}(\Theta^A(\alpha,\cdot))\bigr)$.
Every coordinate below $N$ has a permanent compatible record after some
stage, whereas every record at coordinate $N$ is eventually killed.

\begin{lemma}\label{low2:permanent-residual}
Under the hypotheses of Lemma \ref{low2:auxiliary-partial}, the least
nonsettled child $\beta_i$ eventually keeps one fixed current claim
$(\rho,F)$ permanently; its claimed clopen set is a nonempty residual.
\end{lemma}

\begin{proof}
If the settled child $\beta_j$, $j<i$, has permanent assigned coordinate
$x_j$, then its record is compatible with $A$, so
$x_j\in\operatorname{dom}(\Theta^A(\alpha,\cdot))$. By the definition of
$N$, we therefore have $x_j<N$.

Fix a stage after all of the following conditions hold.
\begin{itemize}
\item The children $\beta_j$, $j<i$, and the higher-priority claim set of
      $\beta_i$ have stabilized.
\item Every fixed finite prefix-free representation of a residual
      $R_s(\beta_i,\rho)$, $\rho\in2^{k_{\beta_i}}$, is resolved by the
      current stage.
\item For every $m<N$, the final least $A$-compatible record selected by the
      ordered search has appeared, so $m$ is never again a gap.
\item No record at a coordinate below $N$ will again be assigned to
      $\beta_i$.
\item Every record at coordinate $N$ that will ever become certified has
      already reached its first certification stage.
\end{itemize}
The resolution condition is legitimate because the higher-priority claim
set is fixed and there are only finitely many bases. The fourth item is
legitimate: if a final record at some $m<N$ is assigned to
$\beta_i$, then it either remains assigned permanently, making $\beta_i$
settled, or it is eventually certified and unassigned. There are only
finitely many $m<N$. The fifth item follows from Lemma
\ref{low2:finite-certification}.

By Lemma \ref{low2:ordered-search-facts}(iii), choose a sufficiently late
stage at which there is no live record at coordinate $N$. By Lemma
\ref{low2:ordered-search-facts}(v), there is then no live record at a higher
coordinate. If the gap results from a pending injury, process it at the next
visit to $\alpha$; injury processing unassigns the affected suffix and places
coordinate $N$ in one-stage cooldown, so no record can be issued at $N$ at
that visit. If there is no pending injury, no record can be issued before the
next visit to $\alpha$. In either case, $N$ remains the least gap until the
next permitted visit.

Call a later visit \emph{permitted} if, after pending injury processing,
$N$ is not in cooldown. Such visits occur arbitrarily late: cooldowns last
only one stage, and
every record at $N$ is eventually killed. Moreover, if $\beta_i$ were
satisfied at every sufficiently late permitted visit, it would be eventually
satisfied. Indeed, after a required success mark is lost, no replacement can
become successful until Case 3 has first created a claim at one visit; that
claim is still present at the next permitted visit. Thus every sufficiently
late failure of satisfaction produces a failure at a permitted visit as well.

At every permitted visit before a new record at $N$ is issued, the settled
children $\beta_j$, $j<i$, retain their assigned coordinates below $N$. The
child $\beta_i$ has no assignment below $N$, and the absence of a live record
at $N$ excludes every higher-coordinate assignment after pending injuries
have been processed. Hence $\beta_i$ is the frontier.

Now consider the first permitted later visit at which $\beta_i$ is not
satisfied. Such a visit exists unless $\beta_i$ is eventually satisfied,
contrary to the hypothesis. If it already has a current claim, denote that
claim by $(\rho,F)$. Otherwise, the resolution condition above ensures that
Case 3 creates a claim $(\rho,F)$ for the least base whose residual is
nonempty and has no valid record.

No settled child $\beta_j$, $j<i$, can have its fixed record certified at
this late stage, since the next visit to $\alpha$ would then unassign it.
Thus no later Case 1 action to the left removes $(\rho,F)$; the claim can be
removed only by a certification of a record assigned to $\beta_i$.

If there is no permitted later visit at which this claim is expansionary,
then $(\rho,F)$ is already permanent. Otherwise, at the first such visit
Case 4 issues a record at the current gap $N$ and assigns it to $\beta_i$.
By our choice of stage, this record cannot be certified. When the record is
killed, the next injury-processing visit unassigns it and leaves the claim
$(\rho,F)$ unchanged. The least gap is again $N$, and the same reasoning
applies to each later record issued at $N$. Because pending injuries are
processed before Case 4, no child to the right of $\beta_i$ can receive a
record at $N$. Hence the claim $(\rho,F)$ is never removed and is permanent.
\end{proof}

\begin{lemma}\label{low2:density}
Let $\alpha$ be a stable true parent assigned to $Q_e$. If $W_e^A$ is dense,
then some child of $\alpha$ is eventually satisfied.
\end{lemma}

\begin{proof}
Suppose no child is eventually satisfied. Let $\beta_i$, $N$, and the
permanent claim $(\rho,F)$ be obtained from Lemmas
\ref{low2:not-all-settle}--\ref{low2:permanent-residual}, and put
$C=R(\rho,F)$. By Lemma \ref{low2:functional}, $\Psi$ is frozen on $C$.

Choose a finite level $L$ which resolves $(\rho,F)$ and is at least the stage
at which the permanent claim was created and its output became frozen. For every
$\tau\in2^L$ with $[\tau]\subseteq C$, choose by density a string $\nu_\tau\in W_e^A$ properly extending $\Psi^\tau$.
There are only finitely many such strings. Choose actual $A$-oracle
enumeration computations witnessing their membership and let $w$ be the
maximum of their uses. If $N>0$, let $b_{N-1}$ be the guard length of the
final least $A$-compatible record selected by the ordered search at $N-1$;
if $N=0$, put $b_{-1}=0$. Set $B=\max\{w,b_{N-1}\}+1$.
Choose $s_0\geq L+1$, after all the stabilization stages used above, such
that $A_s\upharpoonright B=A\upharpoonright B$ for every $s\geq s_0$,
and such that all the chosen witnesses have appeared by stage $s_0$.

By Lemma \ref{low2:ordered-search-facts}(iii), choose a stage after $s_0$ at
which there is no live record at $N$. Process at the next visit to $\alpha$ any
pending injury caused by the loss of the record at $N$. If such an injury is
processed, the one-stage cooldown prevents the issue of a new record at $N$
and hence any new assignment at that visit. If there is no pending injury, no
record can be issued before the next visit to $\alpha$. Let $r$ be the first
subsequent visit at which the resulting one-stage cooldown, if any, has
expired; in the no-injury case, take $r$ to be the next visit. Since only this
local module can issue a record at $N$, $N$ remains the least gap until the
local cases are considered at $r$.

At that point every settled child $\beta_j$, $j<i$, retains its assigned
coordinate below $N$, while $\beta_i$ has no assignment below $N$. Lemma
\ref{low2:ordered-search-facts}(v) and the absence of a live record at $N$
exclude every live record at a higher coordinate; after pending injuries have been
processed, the frontier index is therefore $i$. No record assigned to a
settled $\beta_j$, $j<i$, can be certified at this late stage, and the
permanent claim prevents $\beta_i$ from being satisfied.

For every $\eta\in2^{r-1}$ with
$[\eta]\subseteq C$, put $\tau=\eta\upharpoonright L$. The recursive
definition of $\Psi$ and the freezing on $C$ give $\Psi^\eta=\Psi^\tau$.
Hence the previously chosen $\nu_\tau$ is a valid witness for $\eta$. The
claim is expansionary. Case 1 does not apply because no record assigned to a
settled $\beta_j$, $j<i$, is certified at this late stage; Case 2 does not
apply because $\beta_i$ has the claim $(\rho,F)$; and Case 3 does not apply
because that claim is already present. Therefore Case 4 applies. Since the
construction chooses least-use witnesses, every selected witness has use at
most $w$, and the new
guard length in \eqref{low2:nested-guard} is at most $B$. Its guard is
therefore an initial segment of the true $A$. This record witnesses $ N\in\operatorname{dom}(\Theta^A(\alpha,\cdot))$,
contradicting the definition of $N$.
\end{proof}

A child $\beta_i=\alpha i$ is a \textbf{stable true child} of $\alpha$ if it
is visited infinitely often and, after some stage, no node lexicographically
to its left is visited, $\beta_i$ is never initialized again, every claim
belonging to a node $\gamma<_L\beta_i$ is fixed, and its claim status in the
$\alpha$-module is fixed in the following sense: it is either permanently
satisfied or has one fixed current claim. In the latter case, its
attempt-record assignment need not stabilize.

\begin{lemma}\label{low2:true-child}
Every stable true parent $\alpha$ has a stable true child.
\end{lemma}

\begin{proof}
First suppose some child is eventually satisfied. By Lemma
\ref{low2:frontier}, every assigned child has a nonempty current claim and is
therefore not satisfied, while every child strictly to the right of the
frontier is initialized and not satisfied. Consequently, the satisfied child
is unique and is the frontier.
Work after the stage from which this frontier is satisfied at every visit to
$\alpha$. No record assigned to an earlier child can subsequently be killed
or certified: the resulting injury processing or claim removal would
initialize the frontier and turn off all its success marks, so it would fail
to be satisfied at the next visit. For the same reason, no claim in its
higher-priority claim set can change, and no successful record needed to
witness its satisfaction can be killed or otherwise become invalid. Thus the
frontier index, its residuals, and its witnessing successful records are all
permanent. Case 2 applies at every late visit,
and this child is visited at every late visit to $\alpha$. No smaller child
is visited later. Its satisfied status and empty claim are permanent, and it
is never initialized again. Claims in smaller sibling subtrees remain fixed
because injury processing is local and those subtrees receive no later
visits; claims further to the left were already fixed for the stable true
parent. Together with the stability of $\alpha$, the absence of smaller
outcomes shows that no node lexicographically to the left of this frontier is
visited later. Hence this frontier is a stable true child.

Now suppose no child is eventually satisfied. Let $i$ be the least index of a
nonsettled child. Every $j<i$ is eventually assigned permanently, so all late
frontier outcomes are at least $i$, and the subtrees of smaller children
receive no more visits. By Lemma \ref{low2:permanent-residual}, $\beta_i$ has
a permanent claim. If that claim is eventually never expansionary,
$\beta_i$ remains the frontier at every late visit. Otherwise it repeatedly
has an $A$-incompatible record at $N$ assigned to it. Every such record is eventually killed,
and after the corresponding injury is processed $\beta_i$ is again the
frontier; the visit at which the next record at $N$ is issued has outcome $i$.
Thus outcome $i$ occurs infinitely often, while no smaller outcome occurs
late. No later construction step to the left initializes $\beta_i$, and both
its current claim and all claims to its left are fixed.
The stability of $\alpha$ and the absence of smaller outcomes also imply that
no node lexicographically to the left of $\beta_i$ is visited later.
Therefore $\beta_i$ is a stable true child.
\end{proof}

The root $\lambda$ is a stable true parent by convention. Iterating Lemma
\ref{low2:true-child} constructs the \textbf{true path}, whose nodes
are stable true parents. 

\begin{lemma}\label{low2:satisfied-implies-q}
Let $\alpha$ be the stable true-path node assigned to $Q_e$. If one of its
children is eventually satisfied, then $Q_e$ is satisfied.
\end{lemma}

\begin{proof}
Let $\beta$ be the eventually satisfied child. Its higher-priority claim
set is eventually fixed, since the creation or removal of a claim in that set
would initialize $\beta$ and turn off all of its success marks. Work after
this stabilization and after $\beta$ is satisfied at every visit. For each
fixed nonempty residual, there is then at most one valid successful record:
while such a record is valid, Case 3 cannot select its base again. If its guard
were incompatible with $A$, the record would eventually be killed and its
success mark would be turned off at the next visit. No replacement can become
successful at that visit, since Case 3 must first create the corresponding
claim and only one case is applied. Thus $\beta$ would not be satisfied at
that visit, a contradiction. Hence every successful record witnessing the
late satisfaction of $\beta$ is compatible with $A$ and permanently valid.

Write $\mathcal H(\beta)$ and $R(\beta,\rho)$ for the stabilized values of
$\mathcal H_s(\beta)$ and $R_s(\beta,\rho)$. Fix $X$ such that $\Psi^X$ is
total, and put
$\rho=X\upharpoonright k_\beta$. If $X$ belongs to the fixed union
$\mathcal H(\beta)$ of higher-priority claimed clopen sets, then one of the
corresponding fixed claims freezes $\Psi$ along $X$ by Lemma
\ref{low2:functional}; Lemma \ref{low2:claim-separation} prevents a later
lower-priority claim from changing the output on that region. This
contradicts totality. Hence
$X\in R(\beta,\rho)$.

By eventual satisfaction, there is a permanently valid successful record for
this exact fixed residual. When Case 1 made that record successful,
the construction extended $\Psi$ throughout the residual to strings
enumerated in $W_e^A[s]$. Lemma \ref{low2:certification-safety} and
compatibility of the guard with $A$ make these enumerations correct in
$W_e^A$. Therefore some initial segment of $\Psi^X$ belongs to $W_e^A$. This
is $Q_e$.
\end{proof}

\begin{proof}[Proof of Theorem \ref{low2}]
By Lemma \ref{low2:functional}, the stage construction defines the recursive
functional $\Psi$.
The same proof as Lemma \ref{15} gives $M$, this time we use Lemma \ref{low2:claim-separation}. Iterating Lemma
\ref{low2:true-child} gives a stable true-path node $\alpha$ of every length.
If $\alpha$ is assigned to $Q_e$ and $W_e^A$ is dense, Lemma
\ref{low2:density} gives an eventually satisfied child, and Lemma
\ref{low2:satisfied-implies-q} verifies $Q_e$. Thus every $Q_e$ holds. The
discussion in Subsection \ref{requ} now shows that $A$ is AEWG.
\end{proof}

\section*{Acknowledgements}
The author was supported by Focused Research Grants of Natural Science Foundation of Jiangsu Province No. BK20243060. The author thank Denis Hirschfeldt for useful comments and Yong Liu, Ruofei Xie and Liang Yu for helpful discussions.


\begin{thebibliography}{10}
  \bibitem{CDG}P.~Cholak, R.~Downey, and  N.~Greenberg.
  {\newblock}Low{$_2$} computably enumerable sets have hyperhypersimple
  supersets. {\newblock}\textit{To appear in The Journal of Symbolic
  Logic}.{\newblock}
  
  \bibitem{MR2078930}Natasha~L.~Dobrinen  and  Stephen~G.~Simpson.
  {\newblock}Almost everywhere domination. {\newblock}\textit{J. Symbolic
  Logic}, 69(3):914--922, 2004.{\newblock}
  
  \bibitem{MR2732288}Rodney~G.~Downey  and  Denis~R.~Hirschfeldt.
  {\newblock}\textit{Algorithmic randomness and complexity}.
  {\newblock}Theory and Applications of Computability. Springer, New York,
  2010.{\newblock}
  \bibitem{MR4741732}Denis~R.~Hirschfeldt, Carl~G.~Jockusch, Jr., and 
  Paul~E.~Schupp. {\newblock}Coarse computability, the density metric,
  Hausdorff distances between Turing degrees, perfect trees, and reverse
  mathematics. {\newblock}\textit{J. Math. Log.}, 24(2):0, 2024.{\newblock}

\bibitem{HR}
Denis R. Hirschfeldt and Tiago Royer,
\emph{Diagonal noncomputability and distances between degrees},
to appear.
  
  \bibitem{MR3883780}Benoit Monin. {\newblock}An answer to the Gamma
  question. {\newblock}In \textit{LICS '18---33rd Annual ACM/IEEE Symposium
  on Logic in Computer Science},  page  0. ACM, New York, 2018.{\newblock}
  
  \bibitem{MR2387945}Keng~Meng Ng. {\newblock}On very high degrees.
  {\newblock}\textit{J. Symbolic Logic}, 73(1):309--342, 2008.{\newblock}
  
  \bibitem{MR2548883}Andr{\'e} Nies. {\newblock}\textit{Computability
  and randomness},  volume~51  of \textit{Oxford Logic Guides}.
  {\newblock}Oxford University Press, Oxford, 2009.{\newblock}
  
  \bibitem{MR2351944}Stephen~G.~Simpson. {\newblock}Almost everywhere
  domination and superhighness. {\newblock}\textit{MLQ Math. Log. Q.},
  53(4-5):462--482, 2007.{\newblock}
  
  \bibitem{MR882921}Robert~I.~Soare. {\newblock}\textit{Recursively
  enumerable sets and degrees}. {\newblock}Perspectives in Mathematical Logic.
  Springer-Verlag, Berlin, 1987. {\newblock}A study of computable functions
  and computably generated sets.{\newblock}
\end{thebibliography}
\end{document}